\documentclass[12pt, a4paper]{article}
\usepackage{lipsum}
\usepackage{authblk}
\usepackage{amsfonts}
\usepackage{amssymb}
\usepackage{amscd}
\usepackage{mathrsfs}
\usepackage{amsmath,amssymb, amsthm}
\usepackage{graphicx}
\usepackage{color}
\usepackage{authblk}
\usepackage{amssymb}
\usepackage{indentfirst,latexsym}
\usepackage[bookmarks=false]{hyperref}
\usepackage{amsmath,tikz}

\newtheorem{theorem}{Theorem}[section]
\newtheorem{corollary}[theorem]{Corollary}
\newtheorem{lemma}[theorem]{Lemma}

\theoremstyle{definition}

\usepackage{enumerate}
\long\def\delete#1{}

\usepackage{xcolor}
\usepackage[normalem]{ulem}

\title{Classification of perfect and total perfect codes in generalized Petersen graphs}
\renewcommand{\thefootnote}{\fnsymbol{footnote}}
\author[a]{Xiaomeng Wang}
\author[b]{Junyang Zhang\footnote{Corresponding author}}
\affil[a]{{\small School of Mathematics and Statistics\\ Lanzhou University\\
 Lanzhou, Gansu 730000\\

  P. R. China}}
\affil[b]{{\small School of Mathematical Sciences~\& Chongqing Key Lab of Cognitive Intelligence and Intelligent Finance\\ Chongqing Normal University\\ Chongqing 401331\\ P. R. China}}

\date{}

\begin{document}

\openup 0.5\jot
\maketitle

\renewcommand{\thefootnote}{\fnsymbol{footnote}}
 \footnotetext{E-mail addresses: wangxiaomeng@lzu.edu.cn (Xiaomeng Wang); jyzhang@cqnu.edu.cn (Junyang Zhang)}

\vspace{-10mm}
\begin{abstract}
In a graph $\Gamma$, a perfect code is an independent set $C$ with the property that every vertex not in $C$ is adjacent to a unique vertex in $C$, and a total perfect code is a set $C$ of vertices of $\Gamma$ such that every vertex of $\Gamma$ is adjacent to a unique vertex in $C$. We classify these codes for generalized Petersen graphs.


\end{abstract}

\section{Introduction}
\label{sec:Intr}

All graphs considered in this paper are finite, undirected and simple.
For a graph $\Gamma$, we denote its vertex set and edge set by $V(\Gamma)$ and $E(\Gamma)$ respectively. Let $\Gamma$ be a graph. For any two subsets $A$ and $B$ of $V(\Gamma)$, we use $E(A,B)$ to denote the set of edges of $\Gamma$ with one end in $A$ and the other in $B$. As a shorthand, $E(A)$ is used to denote $E(A, A)$, the set of edges with both ends in $A$.

\medskip
The \emph{distance} between two vertices in $\Gamma$ is the length of a shortest path connecting them or is defined as $\infty$ if no such path exists. A subset $C$ of $V(\Gamma)$ is called \cite{Big} a \emph{perfect $r$-error-correcting code} in $\Gamma$ if every vertex of $\Gamma$ is at distance at most $r$ to exactly one vertex in $C$.  A perfect $1$-error-correcting code is typically called a {\em perfect code}. Equivalently, $C$ is a perfect code if it is an independent set in $\Gamma$ and every vertex  of $\Gamma$ is either in $C$ or adjacent to exactly
one vertex in $C$. A subset $C$ of $V(\Gamma)$ is said to be a \emph{total perfect code} in $\Gamma$ if every vertex of $\Gamma$ is adjacent to exactly
one vertex in $C$ \cite{Zhou2016}. It follows that a total perfect code in $\Gamma$ induces a matching in $\Gamma$ and therefore must have even cardinality.
In graph theory, a perfect code is also known as an efficient or independent perfect dominating set \cite{DSLW16,Le}, while a total perfect code is referred to as an efficient open dominating set \cite{KPY2014}.

\medskip
The concept of perfect $r$-error-correcting codes in graphs was first introduced by Biggs \cite{Big} as a generalization of the classical concept of \emph{perfect $r$-error-correcting code} from coding theory \cite{HK18,MS77}.
In coding theory, a code $C$ over an alphabet $A$ is a subset of $A^{n}$, where $A^{n}$ denotes the $n$-fold Cartesian product $A\times\cdots\times A$, and its elements are called words of length $n$. The Hamming distance between two words is the number of positions in which they differ. A code
$C$ is a perfect \emph{$r$-error-correcting Hamming code} if every word in
$A^{n}$ is within a Hamming distance of at most
$r$ from exactly one codeword in $C$.
\emph{Perfect $r$-error-correcting Lee code} over the alphabet $\mathbb{Z}_{m}$ is defined in a similar way, where $\mathbb{Z}_{m}$ is the ring of integers modulo $m$ and the Lee distance between two words \begin{equation*}
x=(x_{1},x_{2},\cdots,x_{n})~\mbox{and}~y=(y_{1},y_{2},\cdots,y_{n})\in \mathbb{Z}_{m}^{n}
\end{equation*}
 is defined as follows:
\begin{equation*}
  d_{L}(x,y)=\sum\limits_{i=1}^{n}\min(|x_{i}-y_{i}|,m-|x_{i}-y_{i}|).
\end{equation*}
Recall that the \emph{Hamming graph} $H(n,m)$ is the Cartesian product of $n$ copies of the complete graph $K_m$ and the \emph{grid-like graph}
$L(n, m)$ is the Cartesian product of $n$ copies of the $m$-cycle $C_{m}$. From this perspective,  perfect $r$-error-correcting Hamming codes over an alphabet of size $m$ are precisely the perfect $r$-error-correcting codes in $H(n,m)$. Similarly, for $m\geq3$, perfect $r$-error-correcting Lee codes over $\mathbb{Z}_m$ are precisely the perfect $r$-error-correcting codes in $L(n, m)$.

\medskip
The existence of perfect codes under a given metric is a problem of great significance in coding theory. A complete classification of the parameters for which perfect $r$-error-correcting Hamming codes over Galois fields exist was achieved by the early 1970s \cite{Ti1973,vL1971,Zi1973}. In 1970, Golomb and Welch \cite{GW1970} conjectured that for any $n> 2,r > 1$, and $q\geq 2r+1$ there is no $q$-ary perfect $r$-error-correcting Lee code of length $n$. Despite extensive study, this conjecture remains wide open \cite{HK18,LZ2020}.
In \cite{Big}, Biggs showed that the natural setting for the problem of perfect codes is the class of distance-transitive graphs. Since then, considerable research has been devoted to perfect codes in these graphs and, more generally, in distance-regular graphs \cite{B1977,C1987,E1996,HS1975,MZ1995,SE2020}. More recently, perfect codes in other kinds of graphs have also attracted significant attention \cite{M2011,WX2021,WZ2023,Z2024,Z15,Z2008}, alongside the fascinating related concept of total perfect codes \cite{GHT2008,KG2006,KPY2014,Zhou2016}.

\medskip
Current research on perfect and total perfect codes in graph theory primarily follows two approaches: one determines when subgroups constitute such codes in Cayley graphs \cite{CWX2020,HXZ18,MWWZ20,Z2023,ZZ2021},  and the other establishes existence conditions for prominent graph classes \cite{EJM2009,FHKL2025,FHZ2017,KLC2020}. Despite these efforts, a comprehensive classification for specific graph families remains scarce in the literature \cite{KLS2022}.

\medskip
Given an integer $n\geq 3$ and a nonzero element $k\in \mathbb{Z}_n$, the {\em generalized Petersen graph} $\mathrm{GP}(n,k)$ is defined on the set $\{u_i,v_i\mid i\in \mathbb{Z}_n\}$ of $2n$ vertices, with the adjacencies given by $u_{i}\sim u_{i+1}$, $u_{i}\sim v_{i}$, $v_{i}\sim v_{i+k}$ for all $i\in \mathbb{Z}_n$. It was proved in \cite{EJM2009} that $\mathrm{GP}(n,k)$ admits a perfect code if and only if $n\equiv 0\pmod4$ and $k\equiv 1\pmod2$. In this paper, we classify perfect and total perfect codes in $\mathrm{GP}(n,k)$. We mainly obtain the following two theorems.
\begin{theorem}
\label{pcode}
Let $C$ be a subset of the vertex set of $\mathrm{GP}(n,k)$. Then $C$ is a perfect code in $\mathrm{GP}(n,k)$ if and only if $n\equiv0\pmod{4}$, $k\equiv1\pmod{2}$ and
$C=\{u_{4i+j},v_{4i+j+2}\mid i\in \mathbb{Z}_n\}$
 for some $j\in\{0,1,2,3\}$.
\end{theorem}
\begin{theorem}
\label{tpcode}
Let $C$ be a subset of the vertex set of $\mathrm{GP}(n,k)$. Then $C$ is a total perfect code in $\mathrm{GP}(n,k)$ if and only if one of the following two statements holds:
\begin{enumerate}[{\rm(i)}]
  \item $n\equiv0\pmod{3}$, $k\not\equiv0\pmod{3}$ and $C=C_j$ for some $j\in\{0,1,2\}$ where
\begin{equation*}
C_j:=\{u_{3i+j},v_{3i+j}\mid i\in \mathbb{Z}_n\};
\end{equation*}
  \item $n\equiv0\pmod{6}$, $k\equiv\pm1\pmod{6}$ and $C=C'_j$ for some $j\in\{0,1,2,3,4,5\}$ where
\begin{equation*}
C'_j=\{u_{6i+j},u_{6i+j+1},v_{6i+j+3},v_{6i+j+4}\mid i\in \mathbb{Z}_n\}.
\end{equation*}
\end{enumerate}
\end{theorem}

\medskip
The remainder of this paper is structured as follows. In Section \ref{sec:cons}, we construct perfect and total perfect codes in the graph $\mathrm{GP}(n,k)$. The proofs of Theorems \ref{pcode} and \ref{tpcode} are then presented in Sections \ref{sec:pc} and \ref{sec:tpc}, respectively.
\section{Constructions}
\label{sec:cons}
Let $U=\{u_1,u_2,\ldots,u_n\}$ and $V=\{v_1,v_2,\ldots,v_n\}$
for the remainder of the paper, so the vertex set of $\mathrm{GP}(n,k)$
is a disjoint union of $U$ and $V$. In the graph $\mathrm{GP}(n,k)$, $U$ induces an $n$-cycle, $V$ induces an $n$-cycle or disjoint cycles of the same length, and the edge set
$E(U,V)$ is a perfect matching.

\medskip
The following lemma presents constructions of perfect codes in $\mathrm{GP}(n,k)$. The proof is straightforward and thus omitted. A proof for the case $j=1$ can be found in \cite[Lemma 2]{EJM2009}.
\begin{lemma}
\label{epc}
If $n\equiv0\pmod{4}$ and $k\equiv1\pmod{2}$, then for any $j\in\{0,1,2,3\}$, the set
\begin{equation*}
\{u_{4i+j},v_{4i+j+2}\mid i\in \mathbb{Z}_n\}
\end{equation*}
is a perfect code in $\mathrm{GP}(n,k)$.
\end{lemma}
The following two lemmas give constructions of total perfect codes in $\mathrm{GP}(n,k)$.
\begin{lemma}
\label{etpc1}
If $n\equiv0\pmod{3}$ and $k\not\equiv0\pmod{3}$, then for every $j\in\{0,1,2\}$, the set
\begin{equation*}
C_j:=\{u_{3i+j},v_{3i+j}\mid i\in \mathbb{Z}_n\}
\end{equation*}
is a total perfect code in $\mathrm{GP}(n,k)$.
\end{lemma}
\begin{proof}

\medskip
Observe that $C_j$ is a disjoint union of two subsets
\begin{equation*}
\{u_{3i+j}\mid i\in \mathbb{Z}_n\}~\mbox{and}~\{v_{3i+j}\mid i\in \mathbb{Z}_n\}.
\end{equation*}
Clearly, $\{u_{3i+j}\mid i\in \mathbb{Z}_n\}$ is an independent set. The set $\{u_{3i+j}\mid i\in \mathbb{Z}_n\}$ is also independent because $k\not\equiv0\pmod{3}$.
Since $u_{3i+j}$ is adjacent to $v_{3\ell+j}$ if and only if $3i+j=3\ell+j$, it follows that the set $C_j$ induces a $1$-regular subgraph of $\mathrm{GP}(n,k)$.

\medskip
Now, consider an arbitrary vertex $u_{\ell}\in U\setminus C_j$. Then  $\ell\not\equiv j \pmod{3}$, which implies $\ell\equiv j+1 \pmod{3}$ or $\ell\equiv j-1 \pmod{3}$. For the former, $u_{\ell-1}$ is the unique element in $C_j$ adjacent to $u_{\ell}$. For the latter, $u_{\ell+1}$ is the unique element in $C_j$ adjacent to $u_{\ell}$.

\medskip
Similarly, for any $v_{\ell}\in V\setminus C_j$, we also have $\ell\equiv j+1 \pmod{3}$ or $\ell\equiv j-1 \pmod{3}$. Given that $k\not\equiv0\pmod{3}$,  exactly one the two congruences $\ell+k\equiv j \pmod{3}$ and  $\ell-k\equiv j \pmod{3}$ holds.  Let
\begin{equation*}
m=\left\{
\begin{array}{ll}
\ell+k, & \hbox{if}~\ell+k\equiv j \pmod{3}; \\
\ell-k, & \hbox{if}~\ell-k\equiv j \pmod{3}. \\
                    \end{array}
                  \right.
\end{equation*}
Then $v_{m}$ is the unique element in $C_j$ adjacent to $v_{\ell}$.

\medskip
Having proved that $C_j$ induces a $1$-regular subgraph of $\mathrm{GP}(n,k)$ and every element not in $C_j$ is adjacent to a unique element in $C_j$, we conclude that $C_j$ is a total perfect code in $\mathrm{GP}(n,k)$.
\end{proof}
\begin{lemma}
\label{etpc2}
If $n\equiv0\pmod{6}$ and $k\equiv\pm1\pmod{6}$, then for every $j\in\{0,1,2,3,4,5\}$, the set
\begin{equation*}
C'_j=\{u_{6i+j},u_{6i+j+1},v_{6i+j+3},v_{6i+j+4}\mid i\in \mathbb{Z}_n\}
\end{equation*}
is a total perfect code in $\mathrm{GP}(n,k)$.
\end{lemma}
\begin{proof}
We only prove the case $k\equiv 1\pmod{6}$, as the case $k\equiv-1\pmod{6}$ can be handled similarly.

\medskip
Set $U_j=\{u_{6i+j},u_{6i+j+1}\mid i\in \mathbb{Z}_n\}$ and $V_j=\{v_{6i+j+3},v_{6i+j+4}\mid i\in \mathbb{Z}_n\}$. Then $C'_j$ is a disjoint union of $U_j$ and $V_j$. By the definition of $\mathrm{GP}(n,k)$, the set $U_j$ induces a $1$-regular subgraph, and there are no edges between $U_j$ and $V_j$.
Since $k\equiv 1\pmod{6}$, we have $v_{6i+j+3+k}, v_{6i+j+4-k}\in C'_j$ but $v_{6i+j+3-k}, v_{6i+j+4+k}\notin C'_j$.
It follows that $V_j$ also induces a $1$-regular subgraph of $\mathrm{GP}(n,k)$.
Therefore $C'_j$ induces a $1$-regular subgraph of $\mathrm{GP}(n,k)$.

\medskip
To complete the proof, it remains to show that every vertex not in $C'_j$ is adjacent to exactly one vertex in $C'_j$.

\medskip
First consider an arbitrary vertex $u_{\ell}\in U\setminus C'_j$. Then  $\ell\not\equiv j,~j+1 \pmod{6}$. Set
\begin{equation*}
w=\left\{
\begin{array}{ll}
u_{\ell-1}, & \hbox{if}~\ell\equiv j+2 \pmod{6}; \\
v_{\ell}, & \hbox{if}~\ell\equiv j+3~\hbox{or}~j+4 \pmod{6}; \\
u_{\ell+1}, & \hbox{if}~\ell\equiv j+5 \pmod{6}. \\
                    \end{array}
                  \right.
\end{equation*}
Then $w$ is the unique vertex in $C'_j$ adjacent to $u_{\ell}$.

\medskip
Now consider an arbitrary vertex $v_{\ell}\in V\setminus C'_j$. Then $\ell\not\equiv j+3,~j+4 \pmod{6}$. Set
\begin{equation*}
w=\left\{
\begin{array}{ll}
u_{\ell}, & \hbox{if}~\ell\equiv j~\hbox{or}~j+1 \pmod{6}; \\
v_{\ell+k}, & \hbox{if}~\ell\equiv j+2 \pmod{6}; \\
v_{\ell-k}, & \hbox{if}~\ell\equiv j+5 \pmod{6}. \\
                    \end{array}
                  \right.
\end{equation*}
Then $w$ is the unique vertex in $C'_j$ adjacent to $v_{\ell}$.
\end{proof}
\section{Proof of Theorem \ref{pcode}}
\label{sec:pc}
\begin{proof}[Proof of Theorem \ref{pcode}]
By Lemma \ref{epc}, for any $j\in\{0,1,2,3\}$, the set
\begin{equation*}
C:=\{u_{4i+j},v_{4i+j+2}\mid i\in \mathbb{Z}_n\}
\end{equation*}
is a perfect code in $\mathrm{GP}(n,k)$. So the sufficiency holds.

\medskip
We now prove the necessity.  Let $C$ be a perfect code in $\mathrm{GP}(n,k)$.  The proof proceeds in four steps.

\medskip
\textsf{Step 1:} Show that $|C\cap U|=|C\cap V|=\frac{n}{4}$.

\medskip
Since $C$ is a perfect code in the $3$-regular graph $\mathrm{GP}(n,k)$, every vertex in $(U\cup V)\setminus C$ is adjacent exactly one vertex in $C$ and every vertex in $C$ is adjacent to exactly three vertices in $V(\mathrm{GP}(n,k))\setminus C$.
Set $s=|C\cap U|$ and $t=|C\cap V|$. Then $|U\setminus C|=n-s$ and $|V\setminus C|=n-t$.

\medskip
Consider the edges between $C$ and $U\setminus C$.
Observe that every vertex in $C\cap U$ is adjacent to exactly $2$ vertices in $U\setminus C$ and every vertex in $C\cap V$ is adjacent to exactly one vertex in $U\setminus C$. It follows that $|E(C,U\setminus C)|=2s+t$. Since every vertex in $U\setminus C$ is adjacent exactly one vertex in $C$, we get $|E(C,U\setminus C)|=n-s$. Thus $n-s=2s+t$, or equivalently, $3s+t=n$.

\medskip
By a symmetric argument considering the edges between
$C$ and $V\setminus C$, we obtain $s+3t=n$.
Solving the two equations $3s+t=n$ and $s+3t=n$, we get $s=t=\frac{n}{4}$.
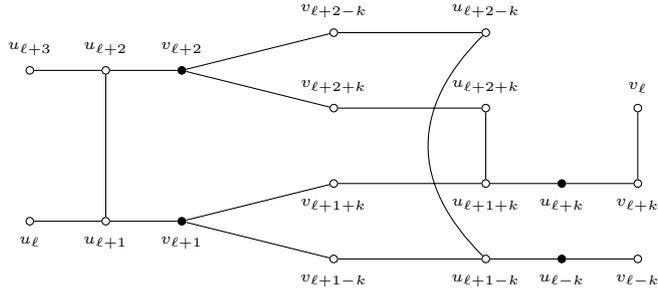
\begin{figure}[t]
  \centering
\begin{tikzpicture}
\tikzstyle{main node}=[draw,shape=circle,label distance=-0.5mm,inner sep=1pt];
\node[main node] (0) at (0,0) [label=below:\tiny{$u_{\ell}$}] {};
\node[main node] (1) at (1,0) [label=below:\tiny{$u_{\ell+1}$}] {};
\node[main node] (2) at (1,2) [label=above:\tiny{$u_{\ell+2}$}] {};
\node[main node] (3) at (0,2) [label=above:\tiny{$u_{\ell+3}$}] {};
\node[main node] (v1) at (2,0) [fill,label=below:\tiny{$v_{\ell+1}$}] {};
\node[main node] (v2) at (2,2) [fill,label=above:\tiny{$v_{\ell+2}$}] {};
\node[main node] (2+k) at (4,1.5) [label=above:\tiny{$v_{\ell+2+k}$}] {};
\node[main node] (2-k) at (4,2.5) [label=above:\tiny{$v_{\ell+2-k}$}] {};
\node[main node] (1+k) at (4,0.5) [label=below:\tiny{$v_{\ell+1+k}$}] {};
\node[main node] (1-k) at (4,-0.5) [label=below:\tiny{$v_{\ell+1-k}$}] {};
\node[main node] (u2+k) at (6,1.5) [label=above:\tiny{$u_{\ell+2+k}$}] {};
\node[main node] (u2-k) at (6,2.5) [label=above:\tiny{$u_{\ell+2-k}$}] {};
\node[main node] (u1+k) at (6,0.5) [label=below:\tiny{$u_{\ell+1+k}$}] {};
\node[main node] (u1-k) at (6,-0.5) [label=below:\tiny{$u_{\ell+1-k}$}] {};
\node[main node] (uk) at (7,0.5) [fill,label=below:\tiny{$u_{\ell+k}$}] {};
\node[main node] (u-k) at (7,-0.5) [fill,label=below:\tiny{$u_{\ell-k}$}] {};
\node[main node] (k) at (8,0.5) [label=below:\tiny{$v_{\ell+k}$}] {};
\node[main node] (l) at (8,1.5) [label=above:\tiny{$v_{\ell}$}] {};
\node[main node] (-k) at (8,-0.5) [label=below:\tiny{$v_{\ell-k}$}] {};
\draw (0)--(1)--(2)--(3)(1)--(v1)(2)--(v2)
(v1)--(1-k)--(u1-k)--(u-k)--(-k)(v1)--(1+k)--(u1+k)--(uk)--(k)--(l)
(v2)--(2-k)--(u2-k)(v2)--(2+k)--(u2+k)
(u1+k)--(u2+k)(u1-k)..controls (5,0.5) and (5,1.5)..(u2-k);
\end{tikzpicture}
\caption{$u_{\ell},u_{\ell+1},u_{\ell+2},u_{\ell+3}\in C$}\label{fig:uell}
\end{figure}

\medskip
\textsf{Step 2:} Show that $C\cap U=\{u_{4i+j}\mid i\in\mathbb{Z}_n\}$ for some $j\in\{0,1,2,3\}$.

\medskip
Assume, for contradiction, that $C\cap U$ does not have this form. Since  $|C\cap U|=\frac{n}{4}$, it follows that $U\setminus C$ contains four consecutive vertices. In other words, there exists $\ell\in\mathbb{Z}_n$ such that $u_{\ell},u_{\ell+1},u_{\ell+2},u_{\ell+3}\notin C$. This leads to  $v_{\ell+k},v_{\ell-k},v_{\ell}\notin C$, as shown in F{\scriptsize IGURE} \ref{fig:uell} where black points represent vertices in $C$, and white  points represent vertices not in $C$. (F{\scriptsize IGURE} \ref{fig:uell} is drawn based on the facts that every black point has no black neighbors, while every white point has exactly one black neighbor.) Since $u_{\ell},v_{\ell+k},v_{\ell-k}\notin C$, no vertex in $C$ is adjacent to $v_{\ell}$. We conclude that $v_{\ell}$ is neither in $C$ nor adjacent to any vertex in $C$, contradicting the assumption that $C$ is a perfect code.

\medskip
\textsf{Step 3:} Show that $C=\{u_{4i+j},v_{4i+j+2}\mid i\in\mathbb{Z}_n\}$ for some $j\in\{0,1,2,3\}$.

\medskip
From Step 2, we have $C\cap U=\{u_{4i+j}\mid i\in\mathbb{Z}_n\}$ for some $j\in\{0,1,2,3\}$. Thus
 \begin{equation*}
u_{4i+j+1},u_{4i+j+2},u_{4i+j+3}\notin C
\end{equation*}
for every $i\in\mathbb{Z}_n$. Now consider the vertex $u_{4i+j+2}$, which is not in $C$ and has three neighbors $u_{4i+j+1},u_{4i+j+3}$ and $v_{4i+j+2}$. Since $u_{4i+j+1},u_{4i+j+3}\notin C$, the only vertex that can be the unique neighbor of $u_{4i+j+2}$ in $C$ is $v_{4i+j+2}$. This implies
 \begin{equation*}
 \{v_{4i+j+2}\mid i\in\mathbb{Z}_n\}\subseteq C\cap V.
\end{equation*}
From Step 1, we get $|C\cap V|=\frac{n}{4}$. Observing
$|\{v_{4i+j+2}\mid i\in\mathbb{Z}_n\}|=\frac{n}{4}$,
it follows that
\begin{equation*}
C\cap V=\{v_{4i+j+2}\mid i\in\mathbb{Z}_n\}.
\end{equation*}
Combining this with the expression for $C\cap U$, we have
\begin{equation*}
C=\{u_{4i+j},v_{4i+j+2}\mid i\in\mathbb{Z}_n\}.
\end{equation*}

\medskip
\textsf{Step 4:} Show that $k\equiv1\pmod2$.

\medskip
From Step 3, there is $j\in\{0,1,2,3\}$ such that
 \begin{equation*}
 C=\{u_{4i+j},v_{4i+j+2}\mid i\in\mathbb{Z}_n\}.
\end{equation*}
Since $v_{4i+j+2+k}$ and $v_{4i+j+2-k}$ are neighbors of $v_{4i+j+2}$, they are not adjacent to any vertex in $C\cap U$. It follows that both $u_{4i+j+2+k}$ and $u_{4i+j+2-k}$ are not in $C$. Therefore $2\pm k\not\equiv0 \pmod4$, that is, $k\equiv1\pmod2$.
\end{proof}

\section{Proof of Theorem \ref{tpcode}}
For a graph $\Gamma$ and two nonnegative integers $a$ and $b$, a subset $C$ of $V(\Gamma)$ is called an $(a,b)$-regular set if $C$ induces an $a$-regular subgraph and every vertex in $V(\Gamma)\setminus C$ is adjacent to exactly $b$ vertices in $C$ \cite{C2019}.
\label{sec:tpc}
\begin{lemma}
\label{para}
Let $C$ be a $(1,b)$-regular set of $\mathrm{GP}(n,k)$, where $b=1,2~\mbox{or}~3$. Then
\begin{equation*}
|E(C)\cap E(U)|=|E(C)\cap E(V)|~\mbox{and}~2|E(C)\cap E(U)|+|E(C)\cap E(U,V)|=\frac{bn}{2+b}.
\end{equation*}
\end{lemma}
\begin{proof}
Since $C$ is a $(1,b)$-regular set of $\mathrm{GP}(n,k)$, the subgraph induced by $C$ is $1$-regular and every vertex in $(U\setminus C)\cup (V\setminus C)$ has exactly $b$ neighbors in $C$. Obviously, every edge in $E(C)$ is incident to exactly $2$ vertices in $U\setminus C$. By considering the edges between
$U\setminus C$ and $C$, we get $b|U\setminus C|=|C|$. Set $|E(C)\cap E(U)|=r$, $|E(C)\cap E(U,V)|=s$ and $|E(C)\cap E(V)|=t$. Then $|C|=2(r+s+t)$ and $|C\cap U|=2r+s$. Since
\begin{equation*}
n=|U|=|C\cap U|+|U\setminus C|,
\end{equation*}
we have
\begin{equation*}
n=2r+s+\frac{2}{b}(r+s+t).
\end{equation*}
Multiplying through by $b$ yields
\begin{equation}\label{|U|}
bn=2br+bs+2(r+s+t)=2(b+1)r+(b+2)s+2t.
\end{equation}
By a symmetric argument considering the edges between
$V\setminus C$ and $C$, we have
\begin{equation*}
b|V\setminus C|=|C|,~~|C\cap V|=2t+s
\end{equation*}
and therefore
\begin{equation}\label{|V|}
bn=2(b+1)t+(b+2)s+2r.
\end{equation}
Subtracting (\ref{|V|}) from (\ref{|U|}) gives
\begin{equation*}
0=2(b+1)(r-t)+2(t-r)=2b(r-t).
\end{equation*}
Since $b\ne 0$, we obtain $r=t$. Substituting $r=t$ into (\ref{|V|}) yields $2r+s=\frac{bn}{2+b}$.
\end{proof}
Observing that $(1,1)$-regular sets and total perfect codes are identical concepts, we obtain the following corollary from Lemma \ref{para}.
\begin{corollary}
\label{parat}
Let $C$ be a total perfect code in $\mathrm{GP}(n,k)$. Then $n\equiv0\pmod 3$,
\begin{equation*}
|E(C)\cap E(U)|=|E(C)\cap E(V)|~\mbox{and}~2|E(C)\cap E(U)|+|E(C)\cap E(U,V)|=\frac{n}{3}.
\end{equation*}
\end{corollary}

\medskip
Let $C$ be a total perfect code in $\mathrm{GP}(n,k)$, and designate its vertices as black, with the remainder being white. In this coloring, the defining property of $C$ ensures every vertex has exactly one black neighbor and two white neighbors. It follows immediately that all vertices at a distance of $1$ or $2$ from any black vertex must be white. This allows us to depict a determined local structure for any two adjacent black vertices. For example, if $u_0,u_1\in C$, then they have a determinable locality as shown in F{\scriptsize IGURE} \ref{fig:u0}.

\medskip
A key step in proving Theorem \ref{tpcode} is established by the following lemma.
\begin{lemma}
\label{nuu}
Let $C$ be a total perfect code in $\mathrm{GP}(n,k)$. If $E(C)\cap E(U)\neq\emptyset$, then $E(C)\cap E(U,V)=\emptyset$.
\end{lemma}
\begin{proof}
Without loss of generality, assume $u_0u_1\in E(C)\cap E(U)$, that is $u_0,u_1\in C$. Then $u_0$ and $u_1$ have a determinable locality as shown in F{\scriptsize IGURE} \ref{fig:u0}.
 In particular, $u_2,v_2\notin C$. Therefore, exactly one of the two vertices $v_{2+k}$ and $v_{2-k}$ is in $C$. Observing  $\mathrm{GP}(n,k)=\mathrm{GP}(n,-k)$,  we may assume, again without loss of generality, that $v_{2+k}\in C$. Since $v_2\notin C$, it follows that exactly one of the two vertices $u_{2+k}$ and $v_{2+2k}$ is in $C$. We will first prove that $u_6,u_7\in C$ in both cases: whether $u_{2+k}\in C$ or $v_{2+2k}\in C$.

 \medskip
 \textsf{Case 1:} $u_{2+k}\in C$.

 \medskip
In this case,  $u_{2+k}$ and $v_{2+k}$ have a determinable locality as shown in F{\scriptsize IGURE} \ref{fig:uv2+k}. From the data on white vertices  presented in F{\scriptsize IGURE} \ref{fig:u0} and F{\scriptsize IGURE} \ref{fig:uv2+k}, we deduce that $v_4, v_{4-k}\in C$. The derivation process is outlined below:
\begin{equation*}
\left.
  \begin{array}{rr}
\left.
  \begin{array}{rr}
   v_0,u_{k}\notin C\Longrightarrow v_{2k}\in C\Longrightarrow u_{1+2k}\notin C\\
v_{2+2k}\notin C
  \end{array}
\right\}\Longrightarrow  u_{3+2k}\in C\Longrightarrow  v_{4+2k}\notin C\\
u_{4+k}\notin C
\end{array}\right\}\Longrightarrow  v_{4}\in C
\end{equation*}
and
\begin{align*}
 \left.
  \begin{array}{rr}
\left.
  \begin{array}{rr}
 v_{1},u_{1+k}\notin C\Longrightarrow v_{1+2k}\in C\Longrightarrow u_{1+3k}\notin C\\
v_{2+3k}\notin C
  \end{array}
\right\}\Longrightarrow u_{3+3k}\in C\Longrightarrow v_{3+2k}\notin C\\
u_{3+k}\notin C
  \end{array}
\right\}& \\
\left.
  \begin{array}{rr}
  \Longrightarrow v_{3}\in C\Longrightarrow u_{4}\notin C\\
u_{2+2k},v_{3+2k}\notin C\Longrightarrow u_{4+2k}\in C
\Longrightarrow v_{4+k}\notin C
  \end{array}
\right\}\Longrightarrow v_{4-k}\in C.
&
\end{align*}
Thus F{\scriptsize IGURE} \ref{fig:v4-k} shows a determinable locality for $v_{4}$ and $v_{4-k}$.

\medskip
Now, based the data on white vertices  presented in F{\scriptsize IGURE} \ref{fig:u0}, F{\scriptsize IGURE} \ref{fig:uv2+k} and F{\scriptsize IGURE} \ref{fig:v4-k}, we deduce that $u_6, u_{7}\in C$ as follows:
\begin{equation*}
\left.
  \begin{array}{rr}
   u_{3+k},v_{4+k}\notin C\Longrightarrow u_{5+k}\in C\Longrightarrow v_{5}\notin C\\
u_{4}\notin C
  \end{array}
\right\}\Longrightarrow  u_{6}\in C
\end{equation*}
and
\begin{align*}
 \left.
  \begin{array}{rr}
\left.
  \begin{array}{rr}
   u_{3},v_{3+k}\notin C\Longrightarrow v_{3-k}\in C\Longrightarrow u_{3-2k}\notin C\\
v_{4-2k}\notin C
  \end{array}
\right\}\Longrightarrow  u_{5-2k}\in C\Longrightarrow  v_{5-k}\notin C
\\
u_{4-k}\notin C
  \end{array}
\right\}& \\
\left.
  \begin{array}{rr}
   \Longrightarrow u_{6-k}\in C\Longrightarrow v_{6}\notin C \\
    u_{5}\notin C
  \end{array}
\right\}\Longrightarrow u_{7}\in C.
&
\end{align*}
\begin{figure}[t]
  \centering
\begin{tikzpicture}
\tikzstyle{main node}=[draw,shape=circle,label distance=-0.5mm,inner sep=1pt];
\node[main node] (-4) at (-4,2) [fill,label=left:\tiny{$u_0$}] {};
\node[main node] (4) at (4,2) [fill,label=right:\tiny{$u_1$}] {};
\node[main node] (-2) at (-2,1) [label=right:\tiny{$u_{-1}$}] {};
\node[main node] (-6) at (-6,1) [label=left:\tiny{$v_0$}] {};
\node[main node] (2) at (2,1) [label=left:\tiny{$u_2$}] {};
\node[main node] (6) at (6,1) [label=right:\tiny{$v_1$}] {};
\node[main node] (-1) at (-1,0) [label=right:\tiny{$u_{-2}$}] {};
\node[main node] (-3) at (-3,0) [label=left:\tiny{$v_{-1}$}] {};
\node[main node] (-5) at (-5,0) [label=right:\tiny{$v_{-k}$}] {};
\node[main node] (-7) at (-7,0) [label=left:\tiny{$v_k$}] {};
\node[main node] (1) at (1,0) [label=left:\tiny{$u_3$}] {};
\node[main node] (3) at (3,0) [label=right:\tiny{$v_2$}] {};
\node[main node] (5) at (5,0) [label=left:\tiny{$v_{1-k}$}] {};
\node[main node] (7) at (7,0) [label=right:\tiny{$v_{1+k}$}] {};
\draw (-4)--(4)
(-4)--(-2)(-4)--(-6)(-2)--(-1)(-2)--(-3)(-6)--(-5)(-6)--(-7)
(4)--(2)(4)--(6)(2)--(1)(2)--(3)(6)--(5)(6)--(7);
\end{tikzpicture}
\caption{$u_0,u_1\in C$}\label{fig:u0}

\bigskip

\begin{tikzpicture}
\tikzstyle{main node}=[draw,shape=circle,label distance=-0.5mm,inner sep=1pt];
\node[main node] (-4) at (-4,2) [fill,label=left:\tiny{$u_{2+k}$}] {};
\node[main node] (4) at (4,2) [fill,label=right:\tiny{$v_{2+k}$}] {};
\node[main node] (-2) at (-2,1) [label=right:\tiny{$u_{1+k}$}] {};
\node[main node] (-6) at (-6,1) [label=left:\tiny{$u_{3+k}$}] {};
\node[main node] (2) at (2,1) [label=left:\tiny{$v_2$}] {};
\node[main node] (6) at (6,1) [label=right:\tiny{$v_{2+2k}$}] {};
\node[main node] (-1) at (-1,0) [label=right:\tiny{$u_{k}$}] {};
\node[main node] (-3) at (-3,0) [label=left:\tiny{$v_{1+k}$}] {};
\node[main node] (-5) at (-5,0) [label=right:\tiny{$v_{3+k}$}] {};
\node[main node] (-7) at (-7,0) [label=left:\tiny{$u_{4+k}$}] {};
\node[main node] (1) at (1,0) [label=left:\tiny{$u_2$}] {};
\node[main node] (3) at (3,0) [label=right:\tiny{$v_{2-k}$}] {};
\node[main node] (5) at (5,0) [label=left:\tiny{$u_{2+2k}$}] {};
\node[main node] (7) at (7,0) [label=right:\tiny{$v_{2+3k}$}] {};
\draw (-4)--(4)
(-4)--(-2)(-4)--(-6)(-2)--(-1)(-2)--(-3)(-6)--(-5)(-6)--(-7)
(4)--(2)(4)--(6)(2)--(1)(2)--(3)(6)--(5)(6)--(7);
\end{tikzpicture}
\caption{$u_{2+k},v_{2+k}\in C$}\label{fig:uv2+k}
\end{figure}

\begin{figure}[t]
  \centering
\begin{tikzpicture}
\tikzstyle{main node}=[draw,shape=circle,label distance=-0.5mm,inner sep=1pt];
\node[main node] (-4) at (-4,2) [fill,label=left:\tiny{$v_{4-k}$}] {};
\node[main node] (4) at (4,2) [fill,label=right:\tiny{$v_4$}] {};
\node[main node] (-2) at (-2,1) [label=right:\tiny{$u_{4-k}$}] {};
\node[main node] (-6) at (-6,1) [label=left:\tiny{$v_{4-2k}$}] {};
\node[main node] (2) at (2,1) [label=left:\tiny{$u_4$}] {};
\node[main node] (6) at (6,1) [label=right:\tiny{$v_{4+k}$}] {};
\node[main node] (-1) at (-1,0) [label=right:\tiny{$u_{3-k}$}] {};
\node[main node] (-3) at (-3,0) [label=left:\tiny{$u_{5-k}$}] {};
\node[main node] (-5) at (-5,0) [label=right:\tiny{$u_{4-2k}$}] {};
\node[main node] (-7) at (-7,0) [label=left:\tiny{$v_{4-3k}$}] {};
\node[main node] (1) at (1,0) [label=left:\tiny{$u_3$}] {};
\node[main node] (3) at (3,0) [label=right:\tiny{$u_5$}] {};
\node[main node] (5) at (5,0) [label=left:\tiny{$u_{4+k}$}] {};
\node[main node] (7) at (7,0) [label=right:\tiny{$v_{4+2k}$}] {};
\draw (-4)--(4)
(-4)--(-2)(-4)--(-6)(-2)--(-1)(-2)--(-3)(-6)--(-5)(-6)--(-7)
(4)--(2)(4)--(6)(2)--(1)(2)--(3)(6)--(5)(6)--(7);
\end{tikzpicture}
\caption{$v_{4-k},v_4\in C$}\label{fig:v4-k}
\end{figure}
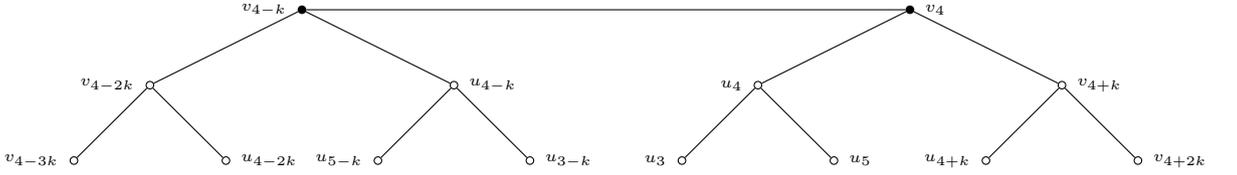

\medskip
\textsf{Case 2:} $v_{2+2k}\in C$.

\medskip
In this case, $v_{2+k}$ and $v_{2+2k}$ have a determinable locality as depicted  in F{\scriptsize IGURE} \ref{fig:v2+k}.
Based on the data from white vertices  presented in F{\scriptsize IGURE} \ref{fig:u0} and F{\scriptsize IGURE} \ref{fig:v2+k},  we derive the following:
\begin{equation*}
\begin{split}
\left.
\begin{array}{rr}
  u_{1+k},v_{k}\notin C\Longrightarrow u_{k-1}\in C\Longrightarrow v_{k-2}\notin C \\
 u_{-2}\notin C
  \end{array}
\right\} \\
\left.
\begin{array}{rr}
\Longrightarrow v_{-2-k}\in C\Longrightarrow u_{-1-k}\notin C\\
 v_{-k}\notin C
  \end{array}
\right\}\Longrightarrow u_{1-k}\in C
\end{split}
\end{equation*}
and
\begin{equation*}
\begin{split}
\left.
  \begin{array}{rr}
   u_{2+k}, v_{1+k}\notin C\Longrightarrow u_{k}\in C\Longrightarrow v_{k-1}\notin C \\
 u_{-1}\notin C
  \end{array}
\right\}\\
 \left.
\begin{array}{rr}
  \Longrightarrow v_{-1-k}\in C\Longrightarrow u_{-k}\notin C \\
 v_{1-k}\notin C
  \end{array}
\right\}\Longrightarrow u_{2-k}\in C.
\end{split}
\end{equation*}
\begin{figure}[t]
  \centering
\begin{tikzpicture}
\tikzstyle{main node}=[draw,shape=circle,label distance=-0.5mm,inner sep=1pt];
\node[main node] (-4) at (-4,2) [fill,label=left:\tiny{$v_{2+k}$}] {};
\node[main node] (4) at (4,2) [fill,label=right:\tiny{$v_{2+2k}$}] {};
\node[main node] (-2) at (-2,1) [label=right:\tiny{$v_{2}$}] {};
\node[main node] (-6) at (-6,1) [label=left:\tiny{$u_{2+k}$}] {};
\node[main node] (2) at (2,1) [label=left:\tiny{$u_{2+2k}$}] {};
\node[main node] (6) at (6,1) [label=right:\tiny{$v_{2+3k}$}] {};
\node[main node] (-1) at (-1,0) [label=right:\tiny{$v_{2-k}$}] {};
\node[main node] (-3) at (-3,0) [label=left:\tiny{$u_{2}$}] {};
\node[main node] (-5) at (-5,0) [label=right:\tiny{$u_{1+k}$}] {};
\node[main node] (-7) at (-7,0) [label=left:\tiny{$u_{3+k}$}] {};
\node[main node] (1) at (1,0) [label=left:\tiny{$u_{1+2k}$}] {};
\node[main node] (3) at (3,0) [label=right:\tiny{$u_{3+2k}$}] {};
\node[main node] (5) at (5,0) [label=left:\tiny{$u_{2+3k}$}] {};
\node[main node] (7) at (7,0) [label=right:\tiny{$v_{2+4k}$}] {};
\draw (-4)--(4)
(-4)--(-2)(-4)--(-6)(-2)--(-1)(-2)--(-3)(-6)--(-5)(-6)--(-7)
(4)--(2)(4)--(6)(2)--(1)(2)--(3)(6)--(5)(6)--(7);
\end{tikzpicture}
\caption{$v_{2+k},v_{2+2k}\in C$}\label{fig:v2+k}

\bigskip

\begin{tikzpicture}
\tikzstyle{main node}=[draw,shape=circle,label distance=-0.5mm,inner sep=1pt];
\node[main node] (-4) at (-4,2) [fill,label=left:\tiny{$u_{1-k}$}] {};
\node[main node] (4) at (4,2) [fill,label=right:\tiny{$u_{2-k}$}] {};
\node[main node] (-2) at (-2,1) [label=right:\tiny{$u_{-k}$}] {};
\node[main node] (-6) at (-6,1) [label=left:\tiny{$v_{1-k}$}] {};
\node[main node] (2) at (2,1) [label=left:\tiny{$u_{3-k}$}] {};
\node[main node] (6) at (6,1) [label=right:\tiny{$v_{2-k}$}] {};
\node[main node] (-1) at (-1,0) [label=left:\tiny{$u_{-1-k}$}] {};
\node[main node] (-3) at (-3,0) [label=left:\tiny{$v_{-k}$}] {};
\node[main node] (-5) at (-5,0) [label=left:\tiny{$v_{1-2k}$}] {};
\node[main node] (-7) at (-7,0) [label=left:\tiny{$v_{1}$}] {};
\node[main node] (1) at (1,0) [label=right:\tiny{$u_{4-k}$}] {};
\node[main node] (3) at (3,0) [label=right:\tiny{$v_{3-k}$}] {};
\node[main node] (5) at (5,0) [label=right:\tiny{$v_{2-2k}$}] {};
\node[main node] (7) at (7,0) [label=right:\tiny{$v_{2}$}] {};
\draw (-4)--(4)
(-4)--(-2)(-4)--(-6)(-2)--(-1)(-2)--(-3)(-6)--(-5)(-6)--(-7)
(4)--(2)(4)--(6)(2)--(1)(2)--(3)(6)--(5)(6)--(7);
\end{tikzpicture}
\caption{$u_{k-1},u_k\in C$}\label{fig:u1-k}
\end{figure}

\medskip
Thus $u_{1-k}$ and $u_{2-k}$ have a determinable locality as depicted in F{\scriptsize IGURE} \ref{fig:u1-k}.
From the data on white vertices  presented in F{\scriptsize IGURE} \ref{fig:u0}, F{\scriptsize IGURE} \ref{fig:v2+k} and F{\scriptsize IGURE} \ref{fig:u1-k}, we further derive:
\begin{equation*}
 \left.
\begin{array}{rr}
  u_{-1-k},v_{-1}\notin C\Longrightarrow v_{-1-2k}\in C\Longrightarrow u_{-2k}\notin C \\
 v_{1-2k}\notin C
  \end{array}
\right\}\Longrightarrow u_{2-2k}\in C\Longrightarrow v_{3-2k}\notin C
\end{equation*}
\begin{equation*}
 \left.
\begin{array}{rr}
  u_{-k},v_{0}\notin C\Longrightarrow v_{-2k}\in C\Longrightarrow u_{1-2k}\notin C \\
 v_{2-2k}\notin C
  \end{array}
\right\}\Longrightarrow u_{3-2k}\in C\Longrightarrow v_{4-2k}\notin C.
\end{equation*}

\medskip
Additionally, we have
\begin{equation*}
\begin{split}
\left.
  \begin{array}{rr}
    u_{3-k},v_{3-2k}\notin C\Longrightarrow v_{3}\in C\Longrightarrow v_{3+2k}\notin C \\
 u_{2+2k}\notin C
  \end{array}
\right\}\Longrightarrow u_{4+2k}\in C\\
\left.
  \begin{array}{rr}
\Longrightarrow v_{4+k}\notin C\\
v_{3}\in C\Longrightarrow u_{4}\notin C
\end{array}\right\}\Longrightarrow  v_{4-k}\in C
\end{split}
\end{equation*}
and
\begin{equation*}
 \left.
   \begin{array}{rr}
    u_{4-k}, v_{4-2k}\notin C\Longrightarrow v_{4}\in C\Longrightarrow v_{4+2k}\notin C \\
u_{3+2k}\notin C
   \end{array}
 \right\}
\Longrightarrow u_{5+2k}\in C\Longrightarrow v_{5+k}\notin C.
\end{equation*}
Since $v_{4-k},v_{4}\in C$, we have that $v_{4}$ and $v_{4-k}$ admit a determinable locality as depicted in F{\scriptsize IGURE} \ref{fig:v4-k}.
Based the data on white vertices  presented in F{\scriptsize IGURE} \ref{fig:v4-k} and F{\scriptsize IGURE} \ref{fig:v2+k}, we deduce that $u_6, u_{7}\in C$ as follows:
\begin{equation*}
\left.
  \begin{array}{rr}
   u_{3+k},v_{4+k}\notin C\Longrightarrow u_{5+k}\in C\Longrightarrow v_{5}\notin C\\
u_{4}\notin C
  \end{array}
\right\}\Longrightarrow  u_{6}\in C
\end{equation*}
and
\begin{equation*}
\left.
  \begin{array}{rr}
   u_{4+k},v_{5+k}\notin C\Longrightarrow u_{6+k}\in C\Longrightarrow v_{6}\notin C\\
u_{5}\notin C
  \end{array}
\right\}\Longrightarrow  u_{7}\in C.
\end{equation*}

\medskip
Having established that $u_{0},u_{1}\in C$ implies $u_{6},u_{7}\in C$,
we can apply the same logic again. Since $u_{6},u_{7}\in C$, it follows that $u_{12},u_{13}\in C$. Applying this rule iteratively, we conclude that $u_{6i},u_{6i+1}\in C$ for all $i\in \mathbb{Z}_n$. Thus $|E(C)\cap E(U)|\geq \frac{n}{6}$. By Lemma \ref{para}, we have
\begin{equation*}
2|E(C)\cap E(U)|+|E(C)\cap E(U,V)|=\frac{n}{3}.
\end{equation*}
Consequently, $|E(C)\cap E(U)|=\frac{n}{6}$ and $|E(C)\cap E(U,V)|=0$.
\end{proof}
\begin{proof}[Proof of Theorem \ref{tpcode}]
The sufficiency follows from Lemmas \ref{etpc1} and \ref{etpc2}. For the necessity, assume that $C$ is a total perfect code in $\mathrm{GP}(n,k)$. The proof proceeds by considering two cases depending on whether $E(C)\cap E(U)$ is empty.

\medskip
\textsf{Case 1:} $E(C)\cap E(U)=\emptyset$.

\medskip
By Corollary \ref{parat}, we have $n\equiv0\pmod 3$, $E(C)\cap E(V)=\emptyset$ and
\begin{equation*}
|E(C)\cap E(U,V)|=|E (C)|=\frac{n}{3}.
\end{equation*}
 This implies that $|C|=\frac{2n}{3}$ and
\begin{equation*}
u_{m}\in C\Longleftrightarrow v_{m}\in C,~\forall m\in \mathbb{Z}_n.
\end{equation*}

\medskip
Now suppose $u_{\ell},v_{\ell}\in C$. Then $v_{\ell+k},u_{\ell+2},u_{\ell+1}\notin C$. Note that $u_{\ell+2}\notin C$ implies $v_{\ell+2}\notin C$.  This together with $v_{\ell+2}\notin C$ leads to $u_{\ell+3}\in C$, and consequently $v_{\ell+3}\in C$. We have thus shown that $u_{\ell},v_{\ell}\in C$ implies $u_{\ell+3},v_{\ell+3}\in C$. By induction,  $u_{\ell+3i},v_{\ell+3i}\in C$ for all $i\in \mathbb{Z}_n$. Thus $C_j\subseteq C$ if $\ell\equiv j\pmod3$ where $j\in\{0,1,2\}$. Observing  that $|C_j|=\frac{2n}{3}=|C|$, we get $C=C_j$. Furthermore, since $v_{\ell+k}\notin C$, we get $k\not\equiv0\pmod3$.

\medskip

\textsf{Case 2:} $E(C)\cap E(U)\neq\emptyset$.

\medskip
By Lemma \ref{nuu}, we have $E(C)\cap E(U,V)=\emptyset$. This, together with Corollary \ref{parat}, implies that $n\equiv0\pmod{6}$ and
\begin{equation*}
|C\cap U|=|C\cap V|=\frac{n}{3}.
\end{equation*}

\medskip
From the proof of Lemma \ref{nuu}, if $u_{0},u_{1}\in C$, then $C\cap U=\{u_{6i},u_{6i+1}\mid i\in \mathbb{Z}_n\}$. By the same logic, for every $\ell\in \mathbb{Z}_n$, if $u_{\ell},u_{\ell+1}\in C$, then $C\cap U=\{u_{\ell+6i},u_{\ell+6i+1}\mid i\in \mathbb{Z}_n\}$. Consequently,
\begin{equation*}
C\cap U=\{u_{j+6i},u_{j+6i+1}\mid i\in \mathbb{Z}_n\}
\end{equation*}
 for some $j\in\{0,1,2,3,4,5\}$.

\medskip
Observe that neither $u_{j+6i+3}$ nor $u_{j+6i+4}$ is adjacent to any vertex in $C\cap U$, which implies
 \begin{equation*}
\{v_{j+6i+3},v_{j+6i+4}\mid i\in \mathbb{Z}_n\}\subseteq C\cap V.
\end{equation*}
Since $|\{v_{j+6i+3},v_{j+6i+4}\mid i\in \mathbb{Z}_n\}|\geq\frac{n}{3}$ and $|C\cap V|=|C\cap U|=\frac{n}{3}$, we get
 \begin{equation*}
 C\cap V=\{v_{j+6i+3},v_{j+6i+4}\mid i\in \mathbb{Z}_n\}.
\end{equation*}
Consequently, $C=C'_j$.
Furthermore, observing $E(C\cap U,C\cap V)=\emptyset$, the set $C\cap V$ must induce a $1$-regular subgraph. Thus, either $v_{j+3+k}\in C'_j$ or $v_{j+3-k}\in C'_j$. The former condition implies $k\equiv1\pmod6$ and the latter implies $k\equiv-1\pmod6$.
\end{proof}

 \medskip
\noindent\textbf{Acknowledgement.}

\medskip
The first author was supported by the Natural Science Foundation of China (No. 12401453), the China Postdoctoral Science Foundation (No. 2024M751251) and the Postdoctoral Fellowship Program of CPSF (No. GZC20240626).
The second author was supported by the Natural Science Foundation of Chongqing (No. CSTB 2025NSCQ-GPX1034).

\end{document}